\documentclass[a4paper,11pt]{article}

\usepackage[top=3.0cm, bottom=3.0cm, inner=3.0cm, outer=3.0cm,
includefoot]{geometry}

\usepackage{verbatim}
\usepackage{caption}
\usepackage{url}
\usepackage{amsmath}
\usepackage{geometry}
\usepackage{amssymb}
\usepackage{amsmath}
\usepackage{graphicx}
\usepackage{amsthm}
\usepackage{bbm}
\usepackage{float}
\usepackage{color,soul}
\usepackage{hyperref}
\usepackage[T1]{fontenc}
\usepackage[utf8]{inputenc}
\usepackage{authblk}
\usepackage{booktabs}
\usepackage{longtable}
\usepackage{enumerate}
\usepackage{tikz}
\usetikzlibrary{arrows}
\usetikzlibrary{decorations.markings}
\usepackage{indentfirst}
\newcommand{\eChar}{\begin{enumerate}[(i)]}
\newcommand{\eCharR}{\begin{enumerate}[(a)]}
\newcommand{\eBr}{\begin{enumerate}[(1)]}

\newcommand\ceil[1]{\left\lceil #1\right\rceil}

\title
{
Connectivity versus Lin-Lu-Yau curvature
}

\author[1]{Kaizhe Chen\thanks{Email: ckz22000259@mail.ustc.edu.cn}}
\author[2]{Shiping Liu\thanks{Email: spliu@ustc.edu.cn}}
\author[3]{Zhe You\thanks{Email: y30231280@mail.ecust.edu.cn}}
\affil[1]{School of Gifted Young, University of Science and Technology of China}
\affil[2]{School of Mathematical Sciences, University of Science and Technology of China}
\affil[3]{School of Mathematics, East China University of Science and Technology}

\date{}

\theoremstyle{plain}
\newtheorem{lemma}{Lemma}[section]
\newtheorem{theorem}[lemma]{Theorem}

\newtheorem{corollary}[lemma]{Corollary}

\theoremstyle{definition}

\newtheorem{claim}{Claim}

\newtheorem{definition}[lemma]{Definition}

\newtheorem{example}[lemma]{Example}

\numberwithin{equation}{section}

\begin{document}

\maketitle

\begin{abstract}
We explore the interaction between connectivity and Lin-Lu-Yau curvature of graphs systematically. The intuition is that connected graphs with large Lin-Lu-Yau curvature also have large connectivity, and vice versa. We prove that the connectivity of a connected graph is lower bounded by the product of its minimum degree and its Lin-Lu-Yau curvature. On the other hand, if the connectivity of a graph $G$ on $n$ vertices is at least $\frac{n-1}{2}$, then $G$ has positive Lin-Lu-Yau curvature. Moreover, the bound $\frac{n-1}{2}$ here is optimal. Furthermore, we prove that the edge-connectivity is equal to the minimum vertex degree for any connected graph with positive Lin-Lu-Yau curvature. As applications, we estimate or determine the connectivity and edge-connectivity of an amply regular graph with parameters $(d,\alpha,\beta)$ such that $1\neq \beta\geq \alpha$. 
\end{abstract}

\section{Introduction}

Connectivity is one of the fundamental concepts of graph theory \cite{BM08, Die17}. The {\it connectivity} $k(G)$ of a non-complete graph $G$ is defined as the minimum number of vertices that require to be removed to disconnect the graph. For consistency, the connectivity of a complete graph $K_n$ with $n$ vertices is defined to be $n-1$. For example, the connectivity of the two graphs depicted in Figure \ref{fig:1} is equal to $1$ and $2$, respectively. 
\begin{figure}[h!]
\centering
\includegraphics[width=.8\textwidth]{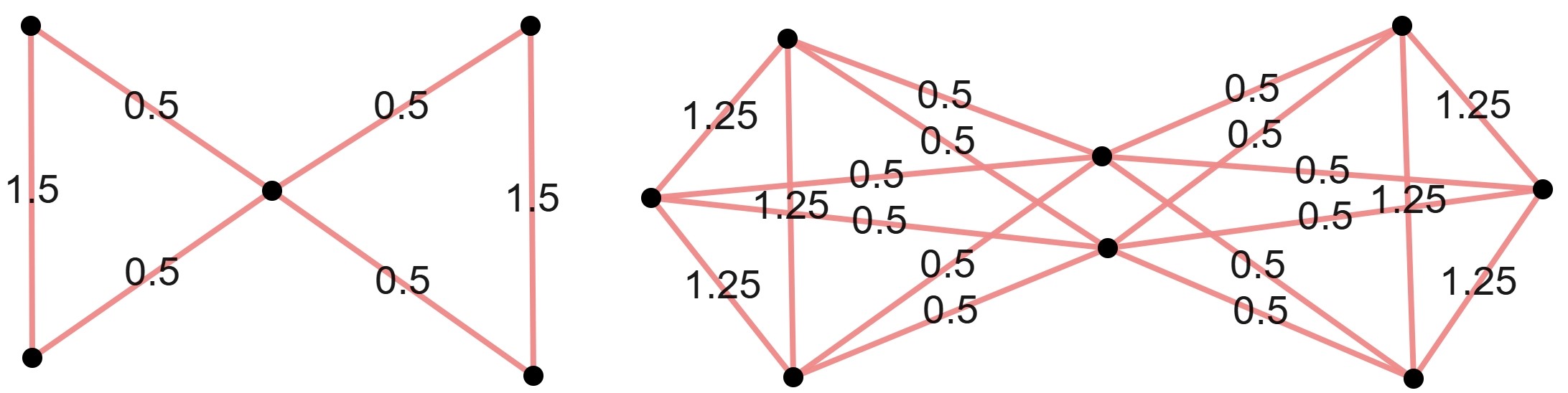}
  \caption{Two graphs labeled with the Lin-Lu-Yau curvature of each edge.}
   \label{fig:1}
\end{figure}

The {\it edge connectivity} $k'(G)$ of a graph $G$ with at least two vertices is defined as the minimum number of edges that required to be removed to disconnect the graph. For consistency, the edge connectivity of a graph with only a single vertex is defined to be $0$.

Let us denote by $\delta(G)$ the minimum vertex degree of a graph $G$. A classical estimate of connectivity of a graph $G$ due to Whitney \cite{Whi32} reads as
\begin{equation}\label{whi}
    k(G)\leq k'(G)\leq \delta(G).
\end{equation}

Both estimates in \eqref{whi} are sharp. For example, the identity  $k(G)=k'(G)=\delta(G)$ holds for any finite connected edge-transitive graph \cite{Mad71, Wat70}, for any finite distance-regular graph \cite{BrHa05, BrKo09, BrMe85} and for any finite $(0,2)$-graph \cite{BrMu97}. For a finite connected vertex-transitive graph $G$, Mader \cite[Satz 6]{Mad71} proved that $k'(G)=\delta(G)$. In general, The three quantities $k(G), k'(G),\delta(G)$ can be arbitrarily far away from each other. Indeed, for any integers $0<\ell\leq m\leq d$, there exists a  graph $G$ with $k(G)=\ell, k'(G)=m$ and $\delta(G)=d$. Even for a finite connected vertex-transitive graph, its connectivity can be strictly smaller than its vertex degree. Mader \cite{Mad71} and Watkins \cite{Wat70} proved independently that $k(G)\geq \frac{2}{3}\left(\delta(G)+1\right)$ for any finite connected vertex-transitive graph $G$. The behavior of connectivity of infinite graphs can be quite different from that of finite graphs \cite{BW80}.


Our first aim is to estimate or determine the connectivity of a connected graph via its Lin-Lu-Yau curvature. In contrast to the above mentioned results with global algebraic symmetric restrictions, the Lin-Lu-Yau curvature detects local structural information of a graph. As a result, our estimates are not only restricted to finite graphs, but also for infinite locally finite graphs.

The curvature notion $\kappa_{LLY}(x,y)$, introduced by Lin, Lu and Yau \cite{LLY11} based on Ollivier's coarse Ricci curvature notion \cite{O09}, is defined on any two distinct vertices $x,y$ of a graph $G$. We call $\kappa_{LLY}(x,y)$ the Lin-Lu-Yau curvature of an edge $\{x,y\}$ whenever $x,y$ are adjacent. The curvature $\kappa_{LLY}(x,y)$ is defined via comparing the two local structures around $x$ and $y$. We have $\kappa_{LLY}(x,y)>0$ if and only if the "distance" between the two local structures around $x$ and $y$ is smaller than the distance between $x$ and $y$. The precise definition is formulated via the Wasserstein distance between probability measures, see Subsection \ref{subsection:WassersteinLLY}. In Figure \ref{fig:1}, the number on each edge is the Lin-Lu-Yau curvature of that edge, which is calculated by the graph curvature calculator \cite{CKLLS22}.

For two vertices $x,y$ in $G$, we denote by $d(x,y)$ the distance between $x$ and $y$. For any positive integer $i$, let us denote the lower Lin-Lu-Yau curvature bound at scale $i$ by 
\[\kappa_{LLY}^{(i)}(G):=\inf_{\substack{x,y\in V\\d(x,y)=i}}\kappa_{LLY}(x,y).\]
When $i=1$, we simply write $\kappa_{LLY}(G)$ instead of $\kappa_{LLY}^{(1)}(G)$. 

Our first result tells that positive Lin-Lu-Yau curvature implies large connectivity.
\begin{theorem}\label{connectivty bound}
Let $G$ be a non-complete connected graph with minimum degree $\delta(G)$ and connectivity $k(G)$. Then
    \begin{align}\notag
        k(G)\geq \delta(G)\kappa_{LLY}^{(2)}(G).
    \end{align}
\end{theorem}

Recall that we have $\kappa_{LLY}^{(2)}(G)\geq \kappa_{LLY}(G)$ due to the triangle inequality of the Wasserstein distance, see \cite[Lemma 2.3]{LLY11} or Lemma \ref{lemma:curv_long_scale} below. Therefore, Theorem \ref{connectivty bound} implies in particular that
\begin{equation}\label{eq:conn_kLLY}
    k(G)\geq \delta(G)\kappa_{LLY}(G).
\end{equation}

For any graph $G$ obtained by deleting a non-empty matching from a complete graph $K_n$ with $n\geq 3$ vertices, the equality holds in the above estimate \eqref{eq:conn_kLLY} and hence also in Theorem \ref{connectivty bound}. Indeed, it was proved independently in \cite[Theorem 1.3]{CS24} and  \cite[Theorem 1]{HHZ25} that $\kappa_{LLY}(G)=1$ holds for such a graph. Therefore, \eqref{eq:conn_kLLY} tells $k(G)\geq \delta(G)$ which forces the equality. 

Furthermore, for the two graphs in Figure \ref{fig:1}, the equality holds in Theorem \ref{connectivty bound}.  In fact, the two graphs in Figure \ref{fig:1} represent an infinite family of graphs for which Theorem \ref{connectivty bound} is sharp. 
\begin{example}[The join of two copies of $K_n$ with a third graph]
Recall that the join $G_1\vee G_2$ of two graphs $G_1$ and $G_2$ is obtained by starting with a disjoint union of two graphs $G_1$ and $G_2$ and adding edges joining every vertex of $G_1$ to every vertex of $G_2$. We consider the join $2K_n\vee G$ of two copies of the complete graph $K_n$ and a third graph $G$. We check that Theorem \ref{connectivty bound} is sharp for $2K_n\vee G$ if the number of vertices in $G$ is smaller than $n$. The two graphs shown in Figure \ref{fig:1} are $2K_2\vee K_1$ and $2K_3\vee 2K_1$, respectively.
\end{example}

There are many more graphs for which Theorem \ref{connectivty bound} is sharp, and even more graphs after rounding up the lower bound to an integer. 
Via estimating the Lin-Lu-Yau curvature of certain \emph{adjacent} vertices directly, we obtain the following estimate. 
\begin{theorem}\label{connectivty bound2}
    Let $G$ be a non-complete connected graph with minimum degree $\delta(G)$ and  connectivity $k(G)$. Suppose that any two adjacent vertices have at most $\alpha$ common neighbors, and any two vertices at distance $2$ have at most $\beta$ common neighbors. Then
    \begin{align}\notag
        k(G)\geq \left(2\kappa_{LLY}(G)+1\right)\delta(G)-2\alpha-\beta-2.
    \end{align}
\end{theorem}
 A $d$-regular graph $G$ is called an \emph{amply regular graph} with parameters $(d,\alpha,\beta)$ if any two adjacent vertices have $\alpha$ common neighbors and any two vertices at distance two have $\beta$ common neighbors. All distance-regular graphs are amply regular. The Hamming graph $H(p,q)$ is the Cartesian product of $p$ copies of the complete graph $K_q$, which is amply regular with parameters $(p(q-1),q-2,2)$.  By \cite[Example 1 and Theorem 3.1]{LLY11}, the Lin-Lu-Yau curvature of each edge of $H(p,q)$  equals $\frac{q}{p(q-1)}$, so Theorem \ref{connectivty bound2} is sharp for all Hamming graphs.
 
 The Lin-Lu-Yau curvature of any amply regular graph $G$ with parameters $(d,\alpha,\beta)$ such that $1\neq \beta\geq \alpha$ has the following lower bound
 \begin{equation}\label{eq:CHLZ}
     \kappa_{LLY}(G)\geq \frac{2+\ceil{\frac{\alpha(\beta-\alpha)}{\beta-1}}}{d}.
 \end{equation}
 This is proved in \cite[Theorem 3.1]{CHLZ24}, improving previous results in \cite{HLX24, LL21}. Combining this result with Theorem \ref{connectivty bound2}, we obtain the following result about connectivity of amply regular graphs, which can be considered as a complement of \cite{Mad71,BrKo09,BrMe85}.
\begin{corollary}
    Let $G$ be an amply regular graph with parameters $(d,\alpha,\beta)$ such that $1\neq \beta\geq \alpha$. Then the connectivity of $G$ satisfies
    \begin{align}\notag
        k(G)\ge d- 2\left\lfloor \frac{\alpha^2-\alpha}{\beta-1} \right\rfloor -\beta +2.
    \end{align}
\end{corollary}

Our next result tells that positive Lin-Lu-Yau curvature determines edge-connectivity. We say a graph $G$ has positive Lin-Lu-Yau curvature, if $\kappa_{LLY}(x,y)>0$ holds for any adjacent vertices $x$ and $y$.
\begin{theorem}\label{edge-connectivity}
Let $G$ be a connected graph with minimum vertex degree $\delta(G)$ and  edge-connectivity $k'(G)$. If $G$ has positive Lin-Lu-Yau curvature, then we have $k'(G)=\delta(G)$. 
\end{theorem}

Graphs with positive Lin-Lu-Yau curvature constitute a big class of graphs. The examples mentioned above 
all belong to this class. This class of graphs also contains infinite graphs, including infinite antitrees with suitable growth properties \cite{DLMP20}.

We can not weaken the assumption of positive Lin-Lu-Yau curvature in Theorem \ref{edge-connectivity} to be non-negative Lin-Lu-Yau curvature. For instance, the Lin-Lu-Yau curvature of an infinite path equals to $0$, while the edge-connectivity is smaller than its vertex degree. On the other hand, the Lin-Lu-Yau curvature of a cycle with at least $5$ vertices is $0$, while the edge-connectivity is still equal to its vertex degree. It is natural to ask whether any \emph{finite} graph with non-negative Lin-Lu-Yau curvature has edge-connectivity $k'(G)=\delta(G)$.

Combining inequality \eqref{eq:CHLZ} with Theorem \ref{edge-connectivity} yields the following interesting corollary. 

 \begin{corollary}
    Any amply regular graph $G$ with parameters $(d,\alpha,\beta)$ such that $1\neq \beta\geq \alpha$ has edge-connectivity $k'(G)=d$.
 \end{corollary}


Conversely, we aim to detect information about Lin-Lu-Yau curvature from the connectivity. Our next result tells that large connectivity implies large Lin-Lu-Yau curvature.

\begin{theorem}\label{LLY lower bound}
    Let $G$ be a graph on $n$ vertices with connectivity $k(G)$ satisfying 
    $$k(G)\ge \frac{n-1}{2}.$$
    Then, for any two adjacent vertices $x$ and $y$ with $d_x\ge d_y$, we have
    \begin{align}\notag
        \kappa_{LLY}(x,y) \ge \frac{2k(G)-n+2}{d_x}.
    \end{align}
\end{theorem}

As a consequence, whenever the connectivity $k(G)$ of a graph $G$ with $n$ vertices is at least $\frac{n-1}{2}$, the Lin-Lu-Yau curvature of $G$ is positive. We notice that the lower bound $\frac{n-1}{2}$ is optimal in the following sense. There exists a graph with $k(G)<\frac{n-1}{2}$ which contains an edge with non-positive Lin-Lu-Yau curvature. Precisely, we have the following result.


\begin{theorem}\label{example}
    Let $n$ and $k$ be two positive integers such that $n-k$ is odd and
    $$\frac{n+1}{3}\le k \le n-1.$$
    Then there exists a graph $G$ with $n$ vertices and connectivity $k$ such that there are two adjacent vertices $x$ and $y$ with $d_x= d_y=k$ satisfying 
    \begin{align}\notag
        \kappa_{LLY}(x,y) = \frac{2k-n+2}{k}.
    \end{align}
\end{theorem}

Theorem \ref{LLY lower bound} and Theorem \ref{example} can be considered as a strong extension of a previous result due to Hehl \cite{Hehl25}. For a graph $G$ with $n$ vertices, Hehl proved that if $\delta(G)\geq \frac{2n}{3}-1$, then $\kappa_{LLY}(G)\geq 0$. Moreover, the bound $\frac{2n}{3}-1$ is optimal.

The Ollivier/Lin-Lu-Yau curvature is a discrete analogue of the Ricci curvature in Riemannian geometry. It has many applications in understanding geometric, analytic, and combinatorial properties of graphs. For more detailed discussions and other discrete curvature notions of graphs, we refer to \cite{NR17} and the references therein. Bakry-\'Emery curvature is another discrete Ricci curvature notion on graphs defined via graph Laplacian and Gamma-calculus. Horn, Purcilly and Stevens \cite{HPS24} established recently the following estimate of connectivity via Bakry-\'Emery curvature: $k(G)\geq \frac{2\kappa_{BE}(G)+\delta(G)+5}{8}$, where $\kappa_{BE}(G)$ is the (non-normalized) Bakry-\'Emery curvature lower bound of $G$. Theorem \ref{connectivty bound} is a counterpart result of their estimate in terms of Lin-Lu-Yau curvature. Notice that the Lin-Lu-Yau curvature and Bakry-\'Emery curvature can behave quite differently. The two curvatures of a given graph can even have opposite signs.

Throughout the paper, we use the following notations.  Let $G=(V,E)$ be an undirected simple locally finite graph and $V,E$ respectively denote the set of vertices and edges of $G$. For a vertex $x$ in $G$, we write $S_1(x)$ for the set of vertices adjacent to $x$, while the degree of $x$ is denoted by $d_x:= |S_1(x)|$.
For a vertex set $X$, we define $S_1(X)$ as 
$$S_1(X):=\left( \bigcup_{x\in X} S_1(x)\right) \backslash X.$$
If two edges have a common end, we call these two edges adjacent. The distance between two vertices $x,y$ is denoted by $d(x,y)$. For two vertex sets $X$ and $Y$, the distance between $X$ and $Y$ is denoted by
$$d(X,Y):=\min\{d(x,y)|x\in X,y\in Y\}.$$
When $X=\{x \}$, we simply write $d(x,Y)$ instead of $d(\{ x\},Y)$. For two disjoint vertex sets $X$ and $Y$, we write $E(X,Y)$ for the set of edges with one end in $X$ and the other in $Y$.

\section{Preliminaries}
In this section, we collect useful results on Lin-Lu-Yau curvature and several important theorems about matching of graphs, including Hall's marriage theorem and K\"onig's theorem on maximum matchings and minimum vertex covers.
\subsection{Wasserstein distance and Lin-Lu-Yau curvature}\label{subsection:WassersteinLLY}
We first recall the definition of Wasserstein distance between probability measures on graphs. 
For a graph $G=(V,E)$, a {\it probability measure} on $G$ is a function $\mu:V\to [0,1]$ such that $\sum_{x\in V}\mu(x)=1$.
\begin{definition}[Wasserstein distance]
     Let $G=(V,E)$ be a locally finite graph. Let $\mu_1, \mu_2$ be two probability measures on $G$. The {\it Wasserstein distance} $W(\mu_1, \mu_2)$ between $\mu_1$ and $\mu_2$ is defined as
     \[W(\mu_1,\mu_2)=\inf_{\pi}\sum_{x,y\in V}d(x,y)\pi(x,y),\]
     where the infimum is taken over all the maps $\pi: V\times V\to [0,1]$ satisfying
     \[\mu_1(x)=\sum_{y\in V}\pi(x,y),\ \forall x\in V;\]
     and
     \[\mu_2(y)=\sum_{x\in V}\pi(x,y),\ \forall y\in V.\] Such a map is called a {\it transport plan} from $\mu_1$ to $\mu_2$. We call a transport plan {\it simple} if, for each vertex $x\in V$, we have
     $$\pi(x,x)=\min\{ \mu_1(x), \mu_2(x)\}.$$
\end{definition}
The Kantorovich duality tells the following identity about Wasserstein distance, see, e.g., \cite[Theorem 1.14]{Vil03}.
$$W(\mu_1, \mu_2)=\sup _f \sum_{x \in V} f(x)\left(\mu_1(x)-\mu_2(x)\right),$$
where the supremum is taken over all $1$-Lipschitz functions $f$, which means
\[
|f(x)-f(y)|\leq d(x,y), \quad \forall\,x,y\in V.
 \]
 
We consider the following particular probability measure around a vertex $x\in V$. For a constant $p\in [0,1]$,
    \[\mu_x^p(y):=\left\{
                    \begin{array}{ll}
                      p, & \hbox{if $y=x$;} \\
                      \frac{1-p}{d_x}, & \hbox{if $xy\in E$;} \\
                      0, & \hbox{otherwise.}
                    \end{array}
                  \right.
     \]     
The Ollivier/Lin-Lu-Yau curvature of graphs is defined as follows.
\begin{definition}[$p$-Ollivier curvature and Lin--Lu--Yau curvature] 
Let $G=(V,E)$ be a locally finite graph. For any two distinct vertices $x$ and $y$, the {\it $p$-Ollivier curvature} $\kappa_p(x,y)$, $p\in [0,1]$, is defined as
     \[\kappa_p(x,y)=1-\frac{W(\mu_x^p,\mu_y^p)}{d(x,y)}.\]
     The {\it Lin-Lu-Yau curvature} $\kappa_{LLY}(x,y)$ is defined as
     \[\kappa_{LLY}(x,y)=\lim_{p\to 1}\frac{\kappa_p(x,y)}{1-p}.\]
\end{definition}
By the linearity of $W(\mu_x^p,\mu_y^p)$ for $p\in \frac{1}{1+\max\{d_x,d_y\}}$, see \cite{BCLMP}, and the property of a Monge problem. We directly have the following expression of $\kappa_{LLY}(x,y)$ when $x$ and $y$ have the same degree. See also \cite[Proposition 2.7]{CKKLMP20} or \cite[Theorem 4.3]{Hehl24}.
\begin{theorem}\label{regular}
    Let $G$ be a locally finite graph. Let $x$ and $y$ be two adjacent vertices of equal degree $d$. Then the Lin-Lu-Yau curvature
    \begin{equation*}\label{def2}
    \kappa_{LLY}(x,y)=\frac{1}{d}\left(d+1-\min_{\substack{\phi: N_x\to N_y}}\sum_{v\in N_x}d(v,\phi(v))\right),
    \end{equation*}
    where the minimum is taken over all the bijections $\phi$ from $N_x:=S_1(x)\setminus (S_1(y)\cup\{y \})$ to $N_y:=S_1(y)\setminus (S_1(x)\cup\{x \})$.
\end{theorem}

We define the lower Lin-Lu-Yau curvature bound of a graph at various scales as follows.
\begin{definition}\label{defn:LLY_lower_bound}
    Let $G=(V,E)$ be a locally finite graph. For any positive integer $i$, we define the lower Lin-Lu-Yau curvature bound of $G$ at scale $i$ as follows:
    \begin{equation*}
       \kappa_{LLY}^{(i)}(G):=\min_{\substack{x,y\in V\\d(x,y)=i}}\kappa_{LLY}(x,y). 
    \end{equation*}
\end{definition}
When $i=1$, we simply write $\kappa_{LLY}(G)$ instead of $\kappa_{LLY}^{(1)}(G)$.
The following lemma is a reformulation of \cite[Lemma 2.3]{LLY11}, which is proved by the triangle inequality of Wasserstein distance.
\begin{lemma}\label{lemma:curv_long_scale}
 Let $G$ be a locally finite graph. For any positive integer $i$, it holds that
       \begin{equation*}
       \kappa_{LLY}^{(i)}(G)\geq \kappa_{LLY}(G). 
    \end{equation*}
\end{lemma}

\subsection{Matching}
Here we recall several important concepts and results about matching from graph theory, see, e.g., \cite{BM08,Die17}. A {\it matching} in a graph is a set of pairwise non-adjacent edges. For a vertex set $A$, a matching {\it of} $A$ is a matching $M$ such that each vertex in $A$ is incident with an edge in $M$. A {\it vertex cover} of an edge set $E$ is a vertex set $X$ such that each edge in $E$ has at least one end in $X$. Below, we collect some well-known theorems, which are key to our proofs.

\begin{theorem}[Hall's Marriage Theorem]\label{Marriage}
   Let $H=(V,E)$ be a bipartite graph with bipartition $V=V_1\sqcup V_2$. Then $H$ has a matching of $V_1$ if and only if
   $$|S_{1}(A)|\geq |A| \,\,\text{for all}\,\,A\subseteq V_1.$$
\end{theorem}

\begin{corollary}\label{Hall}
    Let $H=(V,E)$ be a bipartite graph with bipartition $V=V_1\sqcup V_2$. Suppose that there is a non-negative integer $c$ such that
    $$|S_{1}(A)|\geq |A|-c \,\,\text{for all}\,\,A\subseteq V_1.$$
    Then $H$ contains a matching of size at least $|V_1|-c$.
\end{corollary}

\begin{proof}
    Let $V_3$ be a set of vertices of size $c$. We construct a new bipartite graph with vertex set $V_1\sqcup (V_2\sqcup V_3)$ from $H$ by adding an edge between every vertex of $V_3$ and every vertex of $v$. By Theorem \ref{Marriage}, the new graph contains a matching of $V_1$. Thus, $H$ contains a matching of size at least $|V_1|-c$.
\end{proof}

\begin{theorem}[K\"onig]\label{Konig}
    Let $G$ be a bipartite graph. The maximum cardinality of a matching in $G$ is equal to the minimum cardinality of a vertex cover of its edges.
\end{theorem}

\section{Positive curvature implies large connectivity}\label{Positive Lin-Lu-Yau curvature implies large connectivity}
In this section, we prove sharp lower bounds of connectivity based on Lin-Lu-Yau curvature. Our results tell that a large positive lower bound of Lin-Lu-Yau curvature on each edge (or between any two vertices at distance two) implies large connectivity. Before our proof, we present a basic lemma.

\begin{lemma}\label{neighbor}
    Let $G$ be a non-complete connected graph with connectivity $k(G)$. Let $S\subset V$ be a vertex set of size $|S|=k(G)$ such that $G-S$ is disconnected, and $X$ is a connected component of $G-S$. Then, for any vertex $u\in S$, $u$ has at least one neighbor in $X$.
\end{lemma}

\begin{proof}
    Suppose that $x$ has no neighbors in $X$. Then $G-(S\backslash \{ x \})$ is disconnected, which is contradictory to the minimality of $S$. 
\end{proof}

\begin{proof}[Proof of Theorem \ref{connectivty bound}]
    Let $S\subset V$ be a vertex set of size $|S|=k(G)$ such that $G-S$ is disconnected.
    Let $X$ and $Y$ be two distinct connected components of $G-S$.
    Let $u$ be a vertex in $S$. Let $x$ be a neighbor of $u$ in $X$ and $y$ be a neighbor of $u$ in $Y$ (by Lemma \ref{neighbor}, such $x$ and $y$ must exist).
    Then, we have
    \begin{align*}
        &d(v,X)=1,\,\,\text{for any}\,\,v\in S_1(x)\cap S;\\
        &d(v,X)\geq 1,\,\,\text{for any}\,\,v\in S_1(y)\cap S.
    \end{align*}
    According to Kantorovich duality, we derive
    \begin{align*}
     W(\mu_y^p,\mu_x^p)&= \sup_{\substack{f: V\to \mathbb{R}\\1-\text{Lipschitz}}}\sum_{v\in V}f(v)\left(\mu_y^p(v)-\mu_x^p(v)\right)\\
     &\geq \sum_{v\in V}d(v,X)\mu_{y}^p(v)- \sum_{v\in V}d(v,X)\mu_{x}^p(v)\\
     &\geq  \frac{|S_1(y)\cap S|(1-p)}{d_y}+\left(p+\frac{|S_1(y)\cap Y|(1-p)}{d_y}\right)\times 2-\frac{|S_1(x)\cap S|(1-p)}{d_x}\\
     &= 2-\frac{|S_1(y)\cap S|(1-p)}{d_y}-\frac{|S_1(x)\cap S|(1-p)}{d_x} \\
     &\ge 2-\frac{1-p}{\delta(G)}\left(|S_1(y)\cap S|+ |S_1(x)\cap S|\right).
    \end{align*}
    In the above equality, we have employed the fact that $d_y=|S_1(y)\cap S|+|S_1(y)\cap Y|$.

    By the definition of Lin-Lu-Yau curcature, we estimate
    \begin{align}\label{eq:kLLY}
        \kappa_{LLY}(y,x)=\lim_{p\to 1}\frac{1}{1-p}\left(1-\frac{W(\mu_y^p,\mu_x^p)}{d(x,y)}\right)\leq \frac{|S_1(y)\cap S|+ |S_1(x)\cap S|}{2\delta(G)}.
    \end{align}
Therefore, we derive
\begin{align}\label{eq:bound2}\notag
    k(G)=|S|&\geq \max\{|S_1(x)\cap S|, |S_1(y)\cap S|\}
    \\ &\geq \frac{1}{2}(|S_1(x)\cap S|+|S_1(y)\cap S|)\geq \delta(G)\kappa_{LLY}^{(2)}(G).
\end{align}
We complete the proof.
\end{proof}

Building upon the above proof, we further derive the following estimate. 
\begin{theorem}\label{thm:conn_beta}
Let $G$ be a non-complete connected graph with minimum degree $\delta(G)$ and connectivity $k(G)$.
If any two vertices at distance $2$ have at most $\beta$ common neighbors, then
    \begin{equation}\notag
        k(G)\geq 2\delta(G)\kappa_{LLY}^{(2)}(G)-\beta.
    \end{equation}   
\end{theorem}
\begin{proof} In the proof of Theorem \ref{connectivty bound}, instead of the estimate \eqref{eq:bound2}, we derive from \eqref{eq:kLLY} that
\begin{equation*}\label{eq:bound1}
    k(G)=|S|\geq |S_1(x)\cap S|+|S_1(y)\cap S|-\beta\geq 2\delta(G)\kappa_{LLY}^{(2)}(G)-\beta.
\end{equation*}
This completes the proof.
\end{proof}

Theorem \ref{thm:conn_beta} is sharp for the Cartesian product $K_n\times K_n$. Indeed, this is an amply regular graph with parameters $(d=2(n-1), \alpha=n-2, \beta=2)$. By \cite[Example 1 and Theorem 3.1]{LLY11}, we have $\kappa_{LLY}(G)=\frac{n}{2(n-1)}$. Therefore, the lower bound $2\delta(G)\kappa_{LLY}^{(2)}(G)-\beta\geq 2n-2=d$ forces $k(K_n\times K_n)=d$.

\begin{proof}[Proof of Theorem \ref{connectivty bound2}]
    Let $S\subset V$ be a vertex set of size $|S|=k(G)$ such that $G-S$ is disconnected.
    Let $u$ be a vertex in $S$, and let $d_u^0:=|S_1(u)\cap S|$. Let $X$ and $Y$ be two connected components of $G-S$. By Lemma \ref{neighbor}, there exist $x\in X$ and $y\in Y$ such that both $x$ and $y$ are adjacent to $u$.
    Set 
    $$Z:=\left(\left(S_1(x)\cap X\right)\backslash S_1(u)\right)\cup \{x \}.$$
    According to Kantorovich duality, we derive
    \begin{align*}
     W(\mu_u^p,\mu_x^p)= &\sup_{\substack{f: V\to \mathbb{R}\\1-\text{Lipschitz}}}\sum_{v\in V}f(v)\left(\mu_u^p(v)-\mu_x^p(v)\right)\\
     \geq &\sum_{v\in V}d(v,Z)\mu_{u}^p(v)- \sum_{v\in V}d(v,Z)\mu_{x}^p(v)\\
     \geq &\ p+\frac{d_u^0(1-p)}{d_u}+\frac{(|S_1(u)\cap X|-1)(1-p)}{d_u}+\frac{2(d_u-|S_1(u)\cap X|-d_u^0)(1-p)}{d_u}
     \\ &-\frac{|S_1(x)\cap S|(1-p)}{d_x}-\frac{|S_1(u)\cap S_1(x)\cap X|(1-p)}{d_x}.
    \end{align*}
    Let $\alpha_{ux}^-:=|S_1(u)\cap S_1(x)\cap X|$. By the definition of Lin-Lu-Yau curvature, we have
    \begin{align}\notag
        \kappa_{LLY}(u,x)&=\lim_{p\to 1}\frac{1-W(\mu_u^p,\mu_x^p)}{1-p}
        \\ \label{ux} &\leq \frac{|S_1(u)\cap X|+ d_u^0+1}{d_u}+\frac{|S_1(x)\cap S|+\alpha_{ux}^-}{d_x}-1.
    \end{align}
    Let $\alpha_{uy}^-:=|S_1(u)\cap S_1(y)\cap Y|$. By symmetry, we have
    \begin{align}\label{uy}
        \kappa_{LLY}(u,y) \leq \frac{|S_1(u)\cap Y|+ d_u^0+1}{d_u}+\frac{|S_1(y)\cap S|+\alpha_{uy}^-}{d_y}-1.
    \end{align}
    
    Note that $$|S_1(u)\cap X|+|S_1(u)\cap Y|+ d_u^0\le d_u.$$
    Summing the inequalities \eqref{ux} and \eqref{uy} gives
    \begin{align}\notag
        2\kappa_{LLY}(G)&\le \frac{d_u^0+2}{d_u} +\frac{|S_1(x)\cap S|+\alpha_{ux}^-}{d_x}
        + \frac{|S_1(y)\cap S|+\alpha_{uy}^-}{d_x}-1 
        \\ \label{2K} &\le \frac{1}{\delta}\left( d_u^0+2+ |S_1(x)\cap S|+\alpha_{ux}^- +|S_1(y)\cap S|+\alpha_{uy}^- \right)-1.
    \end{align}
    Let $\alpha_{ux}^+:=|S_1(u)\cap S_1(x)\cap S|$ and $\alpha_{uy}^+:=|S_1(u)\cap S_1(y)\cap S|$.
    Then, we have
    \begin{align}\label{alpha}
        \alpha_{ux}^- +\alpha_{ux}^+ \le \alpha\ 
        {\rm and}\ \alpha_{uy}^- +\alpha_{uy}^+ \le \alpha.
    \end{align}
    According to the inclusion-exclusion principle, we deduce that
    \begin{align}\notag
        k(G)=|S|&\ge |(S_1(u)\cap S)\cup (S_1(x)\cap S)\cup (S_1(y)\cap S)|
        \\ \label{inclusion-exclusion} &\ge d_u^0+ |S_1(x)\cap S|+ |S_1(y)\cap S|-\alpha_{ux}^+ -\alpha_{uy}^+ -\beta.
    \end{align}
    In the above inequality, we have used the assumption that $|S_1(x)\cap S_1(y) \cap S|\le \beta$. 

    Summing up the inequalities \eqref{alpha} and \eqref{inclusion-exclusion}, we have
    $$d_u^0+ |S_1(x)\cap S|+\alpha_{ux}^- +|S_1(y)\cap S|+\alpha_{uy}^- \le k(G)+2\alpha +\beta.$$
    Substituting the above inequality into the  inequality \eqref{2K}, we derive
    \begin{align}\notag
        2\kappa_{LLY}(G)\le \frac{1}{\delta}\left( k(G)+2\alpha +\beta+2 \right)-1,
    \end{align}
    completing the proof.
\end{proof}



\section{Positive curvature determines edge-connectivity}
In this section, we prove Theorem \ref{edge-connectivity}. We first show the following lemma.
\begin{lemma}
    Let $G=(V,E)$ be a connected graph. Let $x$ and $y$ be two adjacent vertices in $G$ with $d_x\ge d_y$. For any $p\in \left[ \frac{1}{1+d_y},1 \right]$, there is a simple optimal transport plan from $\mu_x^p$ to $\mu_y^p$ such that $$\pi(x,y)=p-\frac{1-p}{d_y}.$$ 
\end{lemma}

\begin{proof}
    The existence of a simple optimal transport plan from $\mu_x^p$ to $\mu_y^p$ can be guaranteed by \cite[Corollary 1.16]{Vil03}, see also \cite[Lemma 4.1]{BCLMP}. Let $\pi$ be a simple optimal transport plan from $\mu_x^p$ to $\mu_y^p$ such that $\pi(x,y)$ is maximal. Since $\pi(x,x)=\min\{\mu_x^p(x),\mu_y^p(x)\} =\frac{1-p}{d_y}$, if $\pi(x,y)\ne p-\frac{1-p}{d_y}$, then $\pi(x,y)< p-\frac{1-p}{d_y}$. There is a positive constant $\epsilon$ and two vertices $u,v$ in $V\backslash\{x,y\}$ such that $\pi(x,u)\ge \epsilon$ and $\pi(v,y)\ge \epsilon$. 

    For any vertex $w\in S_1(x)\cap S_1(y)$, we have $\pi(w,w)=\min\{\mu_x^p(w),\mu_y^p(w)\} =\mu_x^p(w)$, and hence $\pi(w,y)=0$. Since $\pi(v,y)>0$, we have $v\in S_1(x)\backslash S_1(y)$.
    For any two vertices $w$ and $z$ in $G$, let $\delta_{wz}: V\times V\to [0,1]$ be a map defined as follows:
    \begin{center}
    $\delta_{wz}(a,b)=\begin{cases}
    1, &{\rm if}\ a=w,b=z;\\
    0, &{\rm otherwise}.
    \end{cases}$
    \end{center}
    Now set $\pi_1:=\pi-\epsilon(\delta_{xu} +\delta_{vy}- \delta_{xy}-\delta_{vu})$. Then, $\pi_1$ is a simple transport plan from $\mu_x^p$ to $\mu_y^p$. If $u\in S_1(x)\cap S_1(y)$, then
    \begin{align}\notag
        \sum_{w,z\in V}d(w,z)\pi(w,z)-\sum_{w,z\in V}d(w,z)\pi_1(w,z)\ge\epsilon(1+2-1-2)=0.
    \end{align}
    If $u\in S_1(y)\backslash S_1(x)$, then we have
    \begin{align}\notag
        \sum_{w,z\in V}d(w,z)\pi(w,z)-\sum_{w,z\in V}d(w,z)\pi_1(w,z)\ge\epsilon(2+2-1-3)=0.
    \end{align}
    Thus, $\pi_1$ is also an optimal transport plan from $\mu_x^p$ to $\mu_y^p$. However, $\pi_1(x,y)\ge \pi(x,y)+\epsilon$, which contradicts the selection of $\pi$.
\end{proof}

\begin{proof}[Proof of Theorem \ref{edge-connectivity}]
Let $G=(V,E)$ be a connected graph with edge-connectivity $k'$ and minimum degree $\delta$.
Suppose that $G$ has positive Lin-Lu-Yau curvature. For a contradiction, we assume that $k'\le \delta -1$. By the definition of edge-connectivity, we can divide $V$ into two non-empty parts $V=X\cup Y$ such that $|E(X,Y)|=k'$. For any two vertices $u$ and $v$ in $G$, we denote by $\alpha_{uv}$ the number of common neighbors of $u$ and $v$. Let $xy$ be an edge in $E(X,Y)$ with $x\in X$ and $y\in Y$ such that
$$\alpha_{xy}=\min\{\alpha_{uv}| uv\in E(X,Y) \}.$$
We divide the discussion into two cases according to whether there is an edge in $E(X,Y)$ that is not adjacent to $xy$.

\noindent\textbf{Case 1:} There is at least one edge in $E(X,Y)$ that is not adjacent to $xy$. 

Without lose of generality, assume that $d_x\ge d_y$. Let us set $N_x:=S_1(x)\backslash(S_1(y)\cup \{y\})$ and $N_y:=S_1(y)\backslash(S_1(x)\cup \{x\})$.
Let $M_1$ be a matching of maximum size in $E(N_x,N_y)$. We construct a bipartite graph $H$ as follows. The vertex set is $V(H):=N_x\cup N_y$. Two vertices $u\in N_x$ and $v\in N_y$ are adjacent in $H$ if and only if the distance between them is at most two in $G$. Let $M_2$ be a matching of maximum size in $H$.

\begin{claim}\label{claim1}
    We have $|M_1|+|M_2|\ge 2d_y-3\alpha_{xy}-3$.
\end{claim}

\begin{proof}
    For $\frac{1}{2}< p<1$, let $\pi$ be a simple optimal transport plan from $\mu_x^p$ to $\mu_y^p$ such that $$\pi(x,y)=p-\frac{1-p}{d_y}.$$
    For $i\in \{1,2,3\}$, denote by $m_i$ the total mass transported with distance at least $i$ under the transport plan $\pi$; that is,
    $$m_i:=\sum_{d(u,v)\ge i}\pi (u,v).$$
    Then the Wasserstein distance between $\mu_x^p$ and $\mu_y^p$ satisfies 
    \begin{align*}
        W(\mu_x^p,\mu_y^p)=m_1+m_2+m_3.
    \end{align*}
    By the selection of $\pi$, we have 
    \begin{align}\label{m1}
        m_1\ge \pi(x,y)+\sum_{u\in N_y}\mu_y^p(u)
        =p-\frac{1-p}{d_y}+\frac{|N_y|(1-p)}{d_y}.
    \end{align}

    Take $c_1:=|N_y|-|M_1|$. By Corollary \ref{Hall} and the definition of $M_1$, there is a subset $A_1\subset N_y$ such that $|S_1(A_1)\cap N_x|\le |A_1|-c_1$. Since $\pi$ is simple, all mass transported to $A_1$ with distance $1$ comes from $S_1(A_1)\cap N_x$. Thus,
    \begin{align}\notag
        m_2&\ge \sum_{u\in A_1}\mu_y^p(u)- \sum_{u\in S_1(A_1)\cap N_x}\mu_x^p(u)
        \\ \notag &=\frac{|A_1|(1-p)}{d_y}- \frac{|S_1(A_1)\cap N_x|(1-p)}{d_x}
        \\ \notag &\ge \frac{|A_1|(1-p)}{d_y}- \frac{(|A_1|-c_1)(1-p)}{d_x}
        \\ \label{m2} &\ge \frac{c_1(1-p)}{d_y}.
    \end{align}

    Similarly, take $c_2:=|N_y|-|M_2|$. By Corollary \ref{Hall} and the construction of the graph $H$, there is a subset $A_2\subset N_y$ such that there are at most $|A_2|-c_2$ vertices $v$ in $N_x$ with $d(v,A_2)\le 2$. By the selection of $\pi$, all mass transported to $A_2$ with distance at most $2$ comes from $N_x$. Thus,
    \begin{align}\notag 
        m_3 &\ge \sum_{u\in A_2}\mu_y^p(u)- \frac{(|A_2|-c_2)(1-p)}{d_x}
        \\ \notag &= \frac{|A_2|(1-p)}{d_y}- \frac{(|A_2|-c_2)(1-p)}{d_x}
        \\ \label{m3} &\ge \frac{c_2(1-p)}{d_y}.
    \end{align}

    Since $W(\mu_x^p,\mu_y^p)$ is linear for $p\in \left[\frac{1}{1+d_x},1\right]$, the condition that $\kappa_{LLY}(x,y)>0$ implies $W(\mu_x^p,\mu_y^p)<1$. Now, combining inequalities \eqref{m1}, \eqref{m2} and \eqref{m3}, we have
    $$d_y>|N_y|+c_1+c_2-1.$$
    The desired result follows by the definition of $c_1,c_2$ and the fact that $N_y=d_y-\alpha_{xy}-1$.
\end{proof}

Let us construct a map $f_1$ from $M_1$ to $E(X,Y)$ as follows. For an edge $uv$ in $M_1$ with $u\in N_x$ and $v\in N_y$, set
\begin{center}
$f_1(uv)=\begin{cases}
xu, &{\rm if}\ u\in Y;\\
vy, &{\rm if}\ u\in X,v\in X;\\
uv, &{\rm otherwise}.
\end{cases}$
\end{center}
Since all the edges in $M_1$ are pairwise non-adjacent, we can directly see that $f_1$ is injective.

Similarly, let us construct an injection $f_2$ from $M_2$ to $E(X,Y)$ as follows. For each $uv\in M_2$ with $u\in N_x$ and $v\in N_y$, by the construction of $H$, the distance between $u$ and $v$ in $G$ must be one or two. If the distance between $u$ and $v$ is two in $G$, let $w$ be a common neighbor of them. Set
\begin{center}
$f_2(uv)=\begin{cases}
xu, &{\rm if}\ u\in Y;\\
vy, &{\rm if}\ u\in X,v\in X;\\
uv, &{\rm if}\ u\in X,v\in Y,uv\in E.\\
uw, &{\rm if}\ u\in X,v\in Y,d(u,v)=2,w\in Y;\\
wv, &{\rm if}\ u\in X,v\in Y,d(u,v)=2,w\in X.
\end{cases}$
\end{center}

For any edge $uv\in E(X,Y)$ with $u\in X$ and $v\in Y$, set $A_{uv}:=S_1(u)\cap S_1(v)$. Let $f_{uv}$ be an injection from $A_{uv}$ to $E(X,Y)$ defined as follows. For any $w\in A_{uv}$, if $w\in X$, take $f_{uv}(w)=wv$. Otherwise, take $f_{uv}(w)=uw$. By the definition, we have $f_1(M_1)\cap f_{xy}(A_{xy})=\emptyset$ and $f_2(M_2)\cap f_{xy}(A_{xy})=\emptyset$.

If $|M_2|=d_y-\alpha_{xy}-1$, then we have
\begin{align}\notag
    |E(X,Y)|&\ge |f_2(M_2)\cup f_{xy}(A_{xy})\cup \{xy\}|
    \\ \notag &=|f_2(M_2)|+|f_{xy}(A_{xy})|+1
    \\ \notag &=|M_2|+|A_{xy}|+1=d_y\ge \delta,
\end{align}
which is contradictory to the assumption that $|E(X,Y)|=k'\le \delta-1$. Now, we assume that $|M_2| \le d_y-\alpha_{xy}-2$. Then, Claim \ref{claim1} shows that
$|M_1|\ge d_y-2\alpha_{xy}-1$.

\begin{claim}\label{claim2}
    Let $uv$ be an edge in $f_1(M_1)$ with $u\in X$ and $v\in Y$. Then $u=x$ or $v=y$.
\end{claim}

\begin{proof}
    Suppose that $u\ne x$ and $v\ne y$. By the definition of $f_1$, we have $u\in N_x$ and $v\in N_y$. Thus, neither $x$ nor $y$ is a common neighbor of $u$ and $v$. Since all the edges in $M_1$ are pairwise non-adjacent, we have 
    $$f_{uv}(A_{uv})\cap \left( f_1(M_1)\cup f_{xy}(A_{xy}) \right)= \emptyset.$$
    By the selection of $xy$, we have $\alpha_{uv}\ge \alpha_{xy}$.
    Since $|M_1|\ge d_y-2\alpha_{xy}-1$, we deduce that
    \begin{align}\notag
        |E(X,Y)|&\ge |f_1(M_1)\cup f_{xy}(A_{xy})\cup \{xy\} \cup f_{uv}(A_{uv})|
        \\ \notag &=|f_1(M_1)|+| f_{xy}(A_{xy})|+1+| f_{uv}(A_{uv})|
        \\ \notag &= |M_1|+|A_{xy}|+1+|A_{uv}|
        \\ \notag &\ge d_y-2\alpha_{xy}-1 +\alpha_{xy}+1+\alpha_{uv}\ge d_y\ge\delta,
    \end{align}
    which is contradictory to the assumption that $|E(X,Y)|=k'\le \delta-1$.
\end{proof}

By the hypothesis of Case $1$, there is an edge $u_0v_0$ in $E(X,Y)$ which is not adjacent to $xy$. Let $w$ be a common neighbor of $u_0$ and $v_0$ such that $w\ne x$ and $w\ne y$. Then, $f_{u_0v_0}(w)\notin f_{xy}(A_{xy})$. In addition, Claim \ref{claim2} implies that $f_{u_0v_0}(w)\notin f_1(M_1)$. Thus, if $x$ and $y$ are not common neighbors of $u_0$ and $v_0$ at the same time, then
$$|f_{u_0v_0}(A_{u_0v_0})\cap \left( f_1(M_1)\cup f_{xy}(A_{xy}) \right)|\le 1.$$
Since $\alpha_{uv}\ge \alpha_{xy}$ and $|M_1|\ge d_y-2\alpha_{xy}-1$, we derive
\begin{align}\notag
        |E(X,Y)|&\ge |f_1(M_1)\cup f_{xy}(A_{xy})\cup \{xy\} \cup f_{u_0v_0}(A_{u_0v_0})\cup \{u_0v_0\}|
        \\ \notag &\ge|f_1(M_1)|+| f_{xy}(A_{xy})|+1+| f_{u_0v}(A_{uv})|+1-1
        \\ \notag &= |M_1|+|A_{xy}|+1+|A_{uv}|
        \\ \notag &\ge d_y-2\alpha_{xy}-1 +\alpha_{xy}+1+\alpha_{uv}\ge d_y\ge\delta,
\end{align}
which is a contradiction.

Now, we may assume that both $x$ and $y$ are common neighbors of $u_0$ and $v_0$. By the definition of $f_2$, we have $u_0v_0\notin f_2(M_2)$. If $|M_2| \ge d_y-\alpha_{xy}-2$, then
\begin{align}\notag
    |E(X,Y)|&\ge |f_2(M_2)\cup f_{xy}(A_{xy})\cup \{xy\}\cup \{u_0v_0\}|
    \\ \notag &=|M_2|+|A_{xy}|+2=d_y\ge \delta.
\end{align}
Thus, we have $|M_2| \le d_y-\alpha_{xy}-3$. Claim \ref{claim1} then shows that $|M_1|\ge d_y-2\alpha_{xy}$. Note that
$$|f_{u_0v_0}(A_{u_0v_0})\cap \left( f_1(M_1)\cup f_{xy}(A_{xy}) \right)|\le 2.$$
We derive
\begin{align}\notag
        |E(X,Y)|&\ge |f_1(M_1)\cup f_{xy}(A_{xy})\cup \{xy\} \cup f_{u_0v_0}(A_{u_0v_0})\cup \{u_0v_0\}|
        \\ \notag &\ge|f_1(M_1)|+| f_{xy}(A_{xy})|+1+| f_{u_0v_0}(A_{u_0v_0})|+1-2
        \\ \notag &= |M_1|+|A_{xy}|+|A_{u_0v_0}|
        \ge d_y-2\alpha_{xy} +\alpha_{xy}+\alpha_{u_0v_0}\ge d_y\ge\delta,
\end{align}
which is a contradiction. We complete the proof of Case $1$.

\noindent\textbf{Case 2:} Every edge in $E(X,Y)\backslash\{xy\}$ is adjacent to $xy$.

\begin{claim}\label{claim3}
    For any edge $uv\in E(X,Y)$ with $u\in X$ and $v\in Y$, there is a vertex $w_1\in X$ such that $w_1$ is adjacent to $u$ but not adjacent to $v$. Similarly, there is a vertex $w_2\in Y$ such that $w_2$ is adjacent to $v$ but not adjacent to $u$.
\end{claim}

\begin{proof}
    Assume that all the neighbors of $u$ in $X$ are adjacent to $v$. Set
    $$E_1:=\{ uw| w\in Y,uw\in E \},$$
    and
    $$E_2:=\{ wv| w\in X,uw\in E \}.$$
    Then, we have
    \begin{align}\notag
        |E(X,Y)|\ge |E_1|+|E_2|=d_u\ge\delta.
    \end{align}
    It contradicts the assumption. Similarly, neighbors of $v$ in $Y$ are not all adjacent to $u$. 
\end{proof}

Now, we divide Case $2$ into two Subcases.

\noindent\textbf{Subcase 1:} We have $|S_1(x)\cap Y|\ge 2$ and $|S_1(y)\cap X|\ge 2$.

Let $u\ne y$ be a neighbor of $x$ in $Y$. We set 
$Z:=\{ z|z\ne x, zu\in E,zx\notin E \}$. Claim \ref{claim3} implies that $Z\ne\emptyset$. By the hypothesis of Case $2$, we have $Z\subset Y$. According to Kantorovich duality, we derive
\begin{align*}
     W(\mu_x^p,\mu_u^p)&= \sup_{\substack{f: V\to \mathbb{R}\\1-\text{Lipschitz}}}\sum_{v\in V}f(v)\left(\mu_x^p(v)-\mu_u^p(v)\right)\\
     &\geq \sum_{v\in V}d(v,Z)\mu_{x}^p(v)- \sum_{v\in V}d(v,Z)\mu_{u}^p(v)\\
     &= \sum_{v\in V}d(v,Z)\mu_{x}^p(v)- p-\frac{2(1-p)}{d_u}-\frac{1-p}{d_u}\sum_{v\in A_{xu}}d(v,Z).
\end{align*}
According to Claim \ref{claim3}, $x$ has a neighbor $w$ in $X$ such that $w$ is not adjacent to $y$. It follows that $d(w,Z)=3$. Set $d_x^+:=|S_1(x)\cap Y|$, then $|S_1(x)\cap X|=d_x-d_x^+$. We deduce that
\begin{align*}
     &\sum_{v\in V}d(v,Z)\mu_{x}^p(v)\\ 
     = &2p+\frac{3(1-p)}{d_x}+\sum_{v\in S_1(x)\cap X\backslash \{ w\}} d(v,Z)\mu_{x}^p(v)+ \sum_{v\in A_{xu}} d(v,Z)\mu_{x}^p(v)+\sum_{v\in S_1(x)\cap Y\backslash A_{xu}} d(v,Z)\mu_{x}^p(v)\\
     \ge &2p+\frac{3(1-p)}{d_x}+\frac{2(d_x-d_x^+-1)(1-p)}{d_x} +\frac{1-p}{d_x}\sum_{v\in A_{xu}}d(v,Z)+ \frac{(d_x^+-\alpha_{xu})(1-p)}{d_x}.
\end{align*}
Therefore, by the definition of Lin-Lu-Yau curvature, we have
\begin{align}\notag
    \kappa_{LLY}(x,u)&=\lim_{p\to 1}\frac{1-W(\mu_x^p,\mu_u^p)}{1-p}
    \\ \notag &\le \frac{2}{d_u}-1+\frac{d_x^+ + \alpha_{xu}-1}{d_x}+\left( \frac{1}{d_u}-\frac{1}{d_x} \right) \sum_{v\in A_{xu}}d(v,Z)
    \\ \label{xu} &\le \frac{2}{\delta}-1+\frac{d_x^+ + \alpha_{xu}-1}{d_x}+\left( \frac{1}{\delta}-\frac{1}{d_x} \right) \sum_{v\in A_{xu}}d(v,Z).
\end{align}
Since $A_{xu}\cup\{u\}\subset S_1(x)\cap Y$, we have $d_x^+ \ge \alpha_{xu}+1$. Then,
$$d_x^+ + \alpha_{xu}-1-\sum_{v\in A_{xu}}d(v,Z)\ge d_x^+ - \alpha_{xu}-1\ge 0.$$
Thus, inequality \eqref{xu} shows that
\begin{align*}
   \kappa_{LLY}(x,u)\le \frac{d_x^+ + \alpha_{xu}+1}{\delta}-1\le \frac{2d_x^+}{\delta}-1.
\end{align*}
By the assumption that $\kappa_{LLY}(x,u)>0$, we derive $d_x^+>\delta/2$. By symmetry, we have $|S_1(y)\cap X|>\delta/2$. It follows that
$$|E(X,Y)|\ge |S_1(x)\cap Y|+|S_1(y)\cap X| -1>\delta-1,$$
which is contradictory to $|E(X,Y)|\le \delta-1$. We complete the proof of Subcase 1. Now, we have either $S_1(y)\cap X= \{x\}$ or $S_1(x)\cap Y= \{y\}$. By symmetry, we only need to consider the case $S_1(y)\cap X= \{x\}$.

\noindent\textbf{Subcase 2:} We have $S_1(y)\cap X= \{x\}$.

Let $Z:=X\backslash\{x \}$. Then Claim \ref{claim3} implies that $Z\ne \emptyset$. Set $d_x^+:=|S_1(x)\cap Y|$. According to Kantorovich duality, we derive
\begin{align*}
     W(\mu_y^p,\mu_x^p)&= \sup_{\substack{f: V\to \mathbb{R}\\1-\text{Lipschitz}}}\sum_{v\in V}f(v)\left(\mu_y^p(v)-\mu_x^p(v)\right)\\
     &\geq \sum_{v\in V}d(v,Z)\mu_{y}^p(v)- \sum_{v\in V}d(v,Z)\mu_{x}^p(v)\\
     &= 2p+\frac{1-p}{d_y}+\frac{2\alpha_{xy}(1-p)}{d_y} +\frac{3(d_y-\alpha_{xy}-1)(1-p)}{d_y}-p-\frac{2d_x^+(1-p)}{d_x}.
\end{align*}
It follows that
\begin{align}\label{4.5}
    0<\kappa_{LLY}(y,x)=\lim_{p\to 1}\frac{1-W(\mu_y^p,\mu_x^p)}{1-p} \le\frac{2d_x^+}{d_x} +\frac{\alpha_{xy}+2}{d_y}-2.
\end{align}
Claim \ref{claim3} implies that $\alpha_{xy}\le d_y-2$. Thus, inequality \eqref{4.5} gives
\begin{align}\label{4.6}
    d_x^+ \ge \frac{d_x+1}{2}.
\end{align}

By Claim \ref{claim3}, $x$ has at least one neighbor $u$ in $X$. Using Kantorovich duality again, we deduce that
\begin{align*}
     W(\mu_u^p,\mu_x^p)&= \sup_{\substack{f: V\to \mathbb{R}\\1-\text{Lipschitz}}}\sum_{v\in V}f(v)\left(\mu_u^p(v)-\mu_x^p(v)\right)\\
     &\geq \sum_{v\in V}d(v,Y)\mu_{u}^p(v)- \sum_{v\in V}d(v,Y)\mu_{x}^p(v)\\
     &= 2p+\frac{1-p}{d_u}+\frac{2\alpha_{xu}(1-p)}{d_u} +\frac{3(d_u-\alpha_{xu}-1)(1-p)}{d_u}-p-\frac{2(d_x-d_x^+)(1-p)}{d_x}.
\end{align*}
It follows that
\begin{align}\label{4.7}
    0<\kappa_{LLY}(u,x)=\lim_{p\to 1}\frac{1-W(\mu_u^p,\mu_x^p)}{1-p} \le \frac{\alpha_{xu}+2}{d_u}-\frac{2d_x^+}{d_x}.
\end{align}
By inequality \eqref{4.6} and the hypothesis that $|E(X,Y)|\le \delta-1$, we have
$$\alpha_{xu}\le d_x-d_x^+\le d_x^+-1\le |E(X,Y)|-1 \le \delta-2.$$
Therefore, inequality \eqref{4.7} shows that
$$0< \frac{\delta}{d_u}-\frac{2d_x^+}{d_x}\le 1-\frac{2d_x^+}{d_x}.$$
This is contradictory to inequality \eqref{4.6}. We complete the proof of Theorem \ref{edge-connectivity}.
\end{proof}

\section{Large connectivity implies large curvature}
In this section, we give a sharp lower bound of Lin-Lu-Yau curvature on each edge. This result shows that large connectivity can guarantee large Lin-Lu-Yau curvature on each edge, which complements the conclusion of Section \ref{Positive Lin-Lu-Yau curvature implies large connectivity}. Before presenting our proof, we need some lemmas to estimate the diameter and Wasserstein distance.

\begin{lemma}\label{diameter}
    Let $G$ be a graph on $n$ vertices. If the minimum degree $\delta(G)$ of $G$ satisfies 
    $$\delta(G)\ge \frac{n-1}{2},$$
    then the diameter of $G$ is at most two.
\end{lemma}
\begin{proof}
    If the diameter of $G$ is at least three, then there exist two non-adjacent vertices $x$ and $y$ with no common neighbors. It follows that
    \begin{align*}
        d_x+d_y=|S_1(x)\cup S_1(y)|\le n-2.
    \end{align*}
    Nevertheless, we have
    \begin{align*}
        d_x+d_y\ge 2\delta(G)\ge n-1,
    \end{align*}
    which is a contradiction.
\end{proof}

\begin{lemma}\label{pi}
    Let $G=(V,E)$ be a graph of diameter two. Let $\mu_1$ and $\mu_2$ be two probability measures on $G$. Suppose that there is a map $\pi_0: V\times V\to [0,1]$ satisfying
    $$\mu_1(x)\ge \sum_{y\in V}\pi_0(x,y),\quad\forall x\in V,$$ 
    and
    $$\mu_2(y)\ge\sum_{x\in V}\pi_0(x,y),\quad\forall y\in V.$$
    Then
    the Wasserstein distance $W(\mu_1, \mu_2)$ between $\mu_1$ and $\mu_2$ satisfies
    \begin{align}\notag
        W(\mu_1,\mu_2)\le 2\left( 1- \sum_{x,y\in V}\pi_0(x,y) \right)+\sum_{x,y\in V}d(x,y)\pi_0(x,y).
    \end{align}
\end{lemma}

\begin{proof}
    Let us consider the following two measures on $G$:
    \begin{align*}
        \mu'_1(x):=\mu_1(x)-\sum_{y\in V}\pi_0(x,y) \quad\forall x\in V,
    \end{align*}
    and
    \begin{align*}
        \mu'_2(y):=\mu_2(y)-\sum_{x\in V}\pi_0(x,y) \quad\forall y\in V.
    \end{align*}
    By the assumption, $\mu'_1$ and $\mu'_2$ are both non-negative and have the same total measure. Since $G$ is connected, there is a transport plan $\pi_1$ between $\mu'_1$ and $\mu'_2$. Set $\pi:=\pi_0+\pi_1$, then $\pi$ is a transport plan between $\mu_1$ and $\mu_2$. Thus,
    \begin{align}\notag
        W(\mu_1,\mu_2)&\le \sum_{x,y\in V}d(x,y)\pi(x,y)
        \\ \notag &=\sum_{x,y\in V}d(x,y)\pi_1(x,y)+\sum_{x,y\in V}d(x,y)\pi_0(x,y)
        \\ \notag &\le 2\sum_{x,y\in V}\pi_1(x,y)+\sum_{x,y\in V}d(x,y)\pi_0(x,y)
        \\ \notag &=2\left( 1- \sum_{x,y\in V}\pi_0(x,y) \right)+\sum_{x,y\in V}d(x,y)\pi_0(x,y).
    \end{align}
    We complete the proof.
\end{proof}

Now, we are prepared to prove Theorem \ref{LLY lower bound}.

\begin{proof}[Proof of Theorem \ref{LLY lower bound}]
    Let $G=(V,E)$ be a graph on $n$ vertices with connectivity $k(G)\ge (n-1)/2$. Let $x$ and $y$ be any two adjacent vertices in $G$ with $d_x\ge d_y$.
    If $G$ is the complete graph, we have $k(G)=n-1$ and $\kappa_{LLY}(x,y)=n/(n-1)$. Then the desired inequality holds.
    
    Now, suppose that $G$ is not complete.  Let us set 
    \begin{align}\notag
        A:=S_1(x)\cap S_1(y),\ N_x:=S_1(x)\setminus (A\cup\{y\}),\ N_y:=S_1(y)\setminus (A\cup\{x\}),
    \end{align}
    and 
    \begin{align}\notag
        B:=V\setminus (S_1(x)\cup S_1(y)).
    \end{align}
    It directly follows that 
    \begin{align}\label{5}
        |A|+|N_x|+|N_y|+|B|+2=n.
    \end{align}
    Using the inequality \eqref{whi}, we have
    \begin{align}\label{2}
        |N_y|+|A|+1=d_y\ge \delta(G)\ge k(G).
    \end{align}
    Let $M$ be a matching between $N_x$ and $N_y$ of maximum size. 
    
    \begin{claim}\label{claim}
        We have $|A|+|B|+|M|+1\ge k(G)$.
    \end{claim}

    \begin{proof}
    If $N_x\cup N_y=\emptyset$, then $|A|+|B|+|M|+1= n-1\ge k(G)$. Now, suppose that $N_x\cup N_y\ne \emptyset$. 
    If $|M|=0$, then there are no edges between $N_x$ and $N_y$. Thus, $G-A\cup B\cup\{ y\}$ is disconnected. By the definition of connectivity, we have $|A|+|B|+1\ge k(G)$. Now, suppose that $|M|>0$.
    According to Theorem \ref{Konig}, there is a vertex set $X$ of size $|M|$ such that there are no edges between $N_x\setminus X$ and $N_y\setminus X$.
    Since $|X|=|M|$, we have $$|(N_x\cup N_y)\setminus X|=|N_x|+|N_y|-|X|\ge |M|+|M|-|M|>0.$$
    Let us set
    \begin{center}
    $Y=\begin{cases}
    A\cup B\cup X\cup \{ x \}, &{\rm if}\ N_x\setminus X\ne \emptyset;\\
    A\cup B\cup X\cup \{ y \}, &{\rm otherwise}.
    \end{cases}$
    \end{center}
    Since there are no edges between $N_x\setminus X$ and $N_y\setminus X$, we deduce that $G-Y$ is disconnected. By the definition of connectivity, we have $|A|+|B|+|M|+1=|Y|\ge k(G)$.
    \end{proof}
    
    Combining the inequalities \eqref{5}, \eqref{2} and Claim \ref{claim}, we derive
    \begin{align}\label{6} 
        (|A|+|N_x|+|N_y|+|B|+2)-(|N_y|+|A|+1)
        -(|A|+|B|+|M|+1)
        \le &n-2k.
    \end{align}
    Let us set $L:=|M|+|A|-|N_x|$. Then the inequality \eqref{6} is transformed into
    \begin{align}\label{L}
        L\ge 2k-n.
    \end{align}
    
    Since $\delta(G)\ge k(G)\ge (n-1)/2$, it follows by Lemma \ref{diameter} that $G$ has diameter two. For any $\frac{1}{1+d_y}\leq p<1$, let $\pi_p: V\times V\to [0,1]$ be the map defined as follows:
    \begin{center}
    $\pi_p(u,v)=\begin{cases}
    p-\frac{1-p}{d_y}, &{\rm if}\ u=x, v=y;\\
    \frac{1-p}{d_y}, &{\rm if}\ u=v=x;\\
    \frac{1-p}{d_x}, &{\rm if}\ u=v,u\in A\cup\{ y\};\\
    \frac{1-p}{d_x}, &{\rm if}\ u\in N_x,v\in N_y,uv\in M;\\
    0, &{\rm otherwise}.
    \end{cases}$
    \end{center}
Then Lemma \ref{pi} shows that 
    \begin{align}\notag
        W(\mu^p_x,\mu^p_y) &\le 2\left( 1- \sum_{u,v\in V}\pi_p(u,v) \right)+\sum_{u,v\in V}d(u,v)\pi_p(u,v)
        \\ \label{W} &= \frac{2(|N_x|-|M|)(1-p)}{d_x}+p-\frac{1-p}{d_y}+ \frac{|M|(1-p)}{d_x}.
    \end{align}
    The inequality \eqref{W} yields
    \begin{align*}\notag
        \frac{1-W(\mu^p_x,\mu^p_y)}{1-p}
        &\ge 1+\frac{|M|-2|N_x|}{d_x}+\frac{1}{d_y}
        \\ &\ge 1+\frac{|M|-2|N_x|}{d_x}+\frac{1}{d_x}=\frac{L+2}{d_x}.
    \end{align*}
    It follows by inequality \eqref{L} that
    \begin{align*}
        \kappa_{LLY}(x,y)=\lim_{p\to 1}\frac{1-W(\mu^p_x,\mu^p_y)}{1-p}\ge \frac{L+2}{d_x} \ge \frac{2k(G)-n+2}{d_x},
    \end{align*}
    completing the proof.
\end{proof}

\section{Sharp examples for Theorem \ref{LLY lower bound}}
In this section, we construct a sequence of graphs, showing that the lower bound on Lin-Lu-Yau curvature given by Theorem \ref{LLY lower bound} is sharp.

\begin{proof}[Proof of Theorem \ref{example}]
    If $k=n-1$, we can let $G$ be the complete graph. Suppose that $k\le n-2$. Let $G=(V,E)$ be the graph defined as follows. The vertex set $V$ is:
    $$V:=\{ x \} \cup \{ y \} \cup N_x \cup N_y \cup A\cup B,$$
    where 
    $$|N_x|=|N_y|=|B|=\frac{n-k-1}{2}\ {\rm and}\ |A|=\frac{3k-n-1}{2}.$$
    The above choices are possible since $n-k$ is odd and $\frac{n+1}{3}\leq k\leq n-1$.
    The edge set $E$ is given by
    $$E:=E_0-E_x\cup E_y\cup E_1,$$
    where
    $$E_0:=\{ uv|u\in V,v\in V \},\ E_x:=\{ xv|v\in B\cup N_y \},\ E_y:=\{ yv|v\in B \cup N_x\},$$
    and
    $$E_1:=\{ uv|u\in N_x,v\in N_y \}.$$
    
    We first show that the connectivity of $G$ is $k$. Otherwise, there is a set $X$ of at most $k-1$ vertices such that $G-X$ is disconnected. Assume that $A\cup B\subset X$. Since $|A|+|B|=k-1$, we have $A\cup B= X$. However, $G-A\cup B$ is connected, which is a contradiction. Now, we suppose that $A\cup B\not\subset X$. Then, there is a vertex $v_0\in (A\cup B)\setminus X$. Note that $v_0$ is adjacent to every vertex in $A\cup B\cup N_x\cup N_y$ (except $v_0$ itself). Therefore, all the vertices in $(A\cup B\cup N_x\cup N_y)\setminus X$ belong to one connected component (denoted by $G_1$) in $G-X$. Since $G-X$ is disconnected, without loss of generality, we assume that $x\notin G_1$. Then, $S_1(x)\setminus \{y \} \subset X$. Since
    $|S_1(x)\setminus \{y \} |=k-1$, we have $S_1(x)\setminus \{y \} = X$. However, $G-S_1(x)\setminus \{y \}$ is connected, which is a contradiction. Thus, the connectivity of $G$ is $k$.

    Now, we calculate the Lin-Lu-Yau curvature of the edge $xy$. Note that, for any $u\in N_x$ and $v\in N_y$, we have $d(x,y)=2$. Since $d_x=d_y=k$, it follows by Theorem \ref{regular} that
    \begin{align}\notag
    \kappa_{LLY}(x,y)&= \frac{1}{k}\left(k+1-\min_{\substack{\phi: N_x\to N_y}}\sum_{v\in N_x}d(v,\phi(v))\right)
    \\ \notag &=\frac{1}{k}\left(k+1-2|N_x|\right)=\frac{2k-n+2}{k},
    \end{align}
    completing the proof.
\end{proof}

\section*{Acknowledgement}
\noindent 
We are very grateful to Huiqiu Lin for inspiring questions and discussions relating connectivity to Lin-Lu-Yau curvature. We thank Jack Koolen warmly for very helpful discussions on connectivity of vertex transitive graphs and amply regular graphs. 
This work is supported by the National Key R \& D Program of China 2023YFA1010200 and the National Natural Science Foundation of China No. 12031017 and No. 12431004. K.C.'s research is supported by the New Lotus Scholars Program PB22000259.

\end{document}